\documentclass[reqno]{amsart}
\usepackage{amsfonts}

\newtheorem{theorem}{Theorem}
\theoremstyle{plain}

\newtheorem{definition}{Definition}

\newtheorem{notation}{Notation}

\newtheorem{proposition}{Proposition}
\newtheorem{remark}{Remark}

\numberwithin{equation}{section}

\input{tcilatex}

\begin{document}

\begin{center}
\textbf{CURVES OF GENERALIZED }$\mathit{AW(}k\mathit{)}$\textbf{-TYPE \\[0pt]
IN EUCLIDEAN SPACES}

\bigskip

\textbf{Kadri ARSLAN, \c{S}aban G\"{U}VEN\c{C}}

\textbf{\ }
\end{center}

\medskip

\textbf{Abstract. }In this study, we consider curves of generalized $AW(k)$%
-type of Euclidean $n$-space. We give curvature conditions of these kind of
curves.

\medskip

 \medskip


\section{\textbf{Introduction}}

In \cite{AW1}, the first author and A. West defined the notion of
submanifolds of $AW(k)$-type. Since then, many works have been done related
to these type of manifolds (for example, see \cite{AO1}, \cite{AO2}, \cite%
{ACDO} and \cite{KA}). In \cite{KA}, the first author and B. K\i l\i \c{c}
studied curves and surfaces of $AW(k)$-type. Further, in \cite{OG1}, C. \"{O}%
zg\"{u}r and F. Gezgin carried out the results for where given in \cite{AO1}
to Bertrand curves and new special curves defined in \cite{IT} by S. Izumiya
and N. Takeuchi. For example, in \cite{AO1} and \cite{KA}, the authors gave
curvature conditions and characterizations related to these curves in $%
\mathbb{R}
^{n}$. Also many results are obtained in Lorentzian spaces in \cite{KT}, 
\cite{KE}, \cite{EMT}, \cite{MBE} and \cite{ME}. In \cite{Yo}, D. Yoon
investigate curvature conditions of curves of $AW(k)$-type in Lie group $G$.
Recently, C. \"{O}zg\"{u}r and the second author studied some types of slant
curves of pseudo-Hermitian $AW(k)$-type in \cite{OG}.

In the present study, we give a generalization of AW(k)-type curves in
Euclidean $n$-space $\mathbb{E}^{n}$. We also give curvature conditions of
these type of curves.

\section{\textbf{Basic Notation}}

Let $\gamma :I\subseteq 
\mathbb{R}
\rightarrow \mathbb{E}^{n}$ be a unit speed curve in $\mathbb{E}^{n}$. The
curve $\gamma $ is called a Frenet curve of osculating order $d$ if its
higher order derivatives $\gamma ^{\prime }(s),\gamma ^{\prime \prime
}(s),...,\gamma ^{\left( d\right) }(s)$ ($d\leq n$) are linearly independent
and $\gamma ^{\prime }(s),\gamma ^{\prime \prime }(s),...,\gamma ^{\left(
d+1\right) }(s)$ are no longer linearly independent for all $s\in I$. To
each Frenet curve of order $d,$ one can associate an orthonormal $d$-frame $%
v_{1},v_{2},...,v_{d}$ along $\gamma $ (such that $\gamma ^{\prime
}(s)=v_{1} $ ) called the Frenet $d$-frame and $(d-1)$ functions $\kappa
_{1},\kappa _{2},...,\kappa _{d-1}:I\rightarrow \mathbb{%
\mathbb{R}
}$ called the Frenet curvatures such that the Frenet formulas are defined in
the usual way:%
\begin{equation}
\left. 
\begin{array}{l}
D_{v_{1}}v_{1}=\kappa _{1}v_{2}, \\ 
D_{v_{1}}v_{2}=-\kappa _{1}v_{1}+\kappa _{2}v_{3}, \\ 
\text{ \  \  \  \  \  \  \  \  \  \  \ }... \\ 
D_{v_{1}}v_{i}=-\kappa _{i-1}v_{i-1}+\kappa _{i}v_{i+1}, \\ 
D_{v_{1}}v_{d}=-\kappa _{d-1}v_{d-1},%
\end{array}%
\right \}  \label{Frenetformulas}
\end{equation}%
where $3\leq i\leq d-1$.

\section{Curves of Generalized $AW(k)$-type}

Let $\gamma $ be a unit speed curve in $n$-dimensional Euclidean space $%
\mathbb{E}^{n}$. By the use of Frenet formulas (\ref{Frenetformulas}), we
obtain the higher order derivatives of $\gamma $ as follows:%
\begin{equation}
\left. 
\begin{array}{l}
\gamma ^{\prime \prime }(s)=\kappa _{1}v_{2}, \\ 
\gamma ^{\prime \prime \prime }(s)=-\kappa _{1}^{2}v_{1}+\kappa _{1}^{\prime
}v_{2}+\kappa _{1}\kappa _{2}v_{3}, \\ 
\gamma ^{(iv)}(s)=-3\kappa _{1}\kappa _{1}^{\prime }v_{1}+\left( \kappa
_{1}^{\prime \prime }-\kappa _{1}^{3}-\kappa _{1}\kappa _{2}^{2}\right) v_{2}
\\ 
\text{ \  \  \  \  \  \  \  \  \  \ }+\left( 2\kappa _{1}^{\prime }\kappa _{2}+\kappa
_{1}\kappa _{2}^{\prime }\right) v_{3}+\kappa _{1}\kappa _{2}\kappa
_{3}v_{4}, \\ 
\gamma ^{(v)}(s)=\left[ -3(\kappa _{1}^{\prime })^{2}-4\kappa _{1}\kappa
_{1}^{\prime \prime }+\kappa _{1}^{4}+\kappa _{1}^{2}\kappa _{2}^{2}\right]
v_{1} \\ 
\text{ \  \  \  \  \  \  \  \  \ }+\left( \kappa _{1}^{\prime \prime \prime
}-6\kappa _{1}^{2}\kappa _{1}^{\prime }-3\kappa _{1}^{\prime }\kappa
_{2}^{2}-3\kappa _{1}\kappa _{2}\kappa _{2}^{\prime }\right) v_{2} \\ 
\text{ \  \  \  \  \  \  \  \  \ }+\left( 3\kappa _{1}^{\prime \prime }\kappa
_{2}+3\kappa _{1}^{\prime }\kappa _{2}^{\prime }-\kappa _{1}^{3}\kappa
_{2}-\kappa _{1}\kappa _{2}^{3}+\kappa _{1}\kappa _{2}^{\prime \prime
}-\kappa _{1}\kappa _{2}\kappa _{3}^{2}\right) v_{3} \\ 
\text{ \  \  \  \  \  \  \  \  \ }+(3\kappa _{1}^{\prime }\kappa _{2}\kappa
_{3}+2\kappa _{1}\kappa _{2}^{\prime }\kappa _{3}+\kappa _{1}\kappa
_{2}\kappa _{3}^{\prime })v_{4}+\kappa _{1}\kappa _{2}\kappa _{3}\kappa
_{4}v_{5}.%
\end{array}%
\right \}  \label{higherderivativesofgamma}
\end{equation}

\begin{notation}
Let us write%
\begin{equation}
\left. 
\begin{array}{l}
N_{1}=\kappa _{1}v_{2}, \\ 
N_{2}=\kappa _{1}^{\prime }v_{2}+\kappa _{1}\kappa _{2}v_{3}, \\ 
N_{3}=\lambda _{2}v_{2}+\lambda _{3}v_{3}+\lambda _{4}v_{4}, \\ 
N_{4}=\mu _{2}v_{2}+\mu _{3}v_{3}+\mu _{4}v_{4}+\mu _{5}v_{5},%
\end{array}%
\right \}  \label{Normalparts}
\end{equation}%
where%
\begin{equation}
\begin{array}{l}
\lambda _{2}=\kappa _{1}^{\prime \prime }-\kappa _{1}^{3}-\kappa _{1}\kappa
_{2}^{2}, \\ 
\lambda _{3}=2\kappa _{1}^{\prime }\kappa _{2}+\kappa _{1}\kappa
_{2}^{\prime }, \\ 
\lambda _{4}=\kappa _{1}\kappa _{2}\kappa _{3}%
\end{array}
\label{lamdas}
\end{equation}%
and%
\begin{equation}
\begin{array}{l}
\mu _{2}=\kappa _{1}^{\prime \prime \prime }-6\kappa _{1}^{2}\kappa
_{1}^{\prime }-3\kappa _{1}^{\prime }\kappa _{2}^{2}-3\kappa _{1}\kappa
_{2}\kappa _{2}^{\prime }, \\ 
\mu _{3}=3\kappa _{1}^{\prime \prime }\kappa _{2}+3\kappa _{1}^{\prime
}\kappa _{2}^{\prime }-\kappa _{1}^{3}\kappa _{2}-\kappa _{1}\kappa
_{2}^{3}+\kappa _{1}\kappa _{2}^{\prime \prime }-\kappa _{1}\kappa
_{2}\kappa _{3}^{2}, \\ 
\mu _{4}=3\kappa _{1}^{\prime }\kappa _{2}\kappa _{3}+2\kappa _{1}\kappa
_{2}^{\prime }\kappa _{3}+\kappa _{1}\kappa _{2}\kappa _{3}^{\prime }, \\ 
\mu _{5}=\kappa _{1}\kappa _{2}\kappa _{3}\kappa _{4}%
\end{array}
\label{nus}
\end{equation}%
are differentiable functions.
\end{notation}

We give the following definition:

\begin{definition}
\label{definitionAW(k)}Frenet curves are

$i)$ \ of generalized $AW(1)$-type if they satisfy $N_{4}=0$,

$ii)$ of generalized $AW(2)$-type if they satisfy%
\begin{equation}
\left \Vert N_{2}\right \Vert ^{2}N_{4}=\left \langle N_{2},N_{4}\right \rangle
N_{2},  \label{eqGAW(2)}
\end{equation}

$iii)$ of generalized $AW(3)$-type if they satify%
\begin{equation}
\left \Vert N_{1}\right \Vert ^{2}N_{4}=\left \langle N_{1},N_{4}\right \rangle
N_{1},  \label{eqGAW3}
\end{equation}

$iv)$ of generalized $AW(4)$-type if they satify%
\begin{equation}
\left \Vert N_{3}\right \Vert ^{2}N_{4}=\left \langle N_{3},N_{4}\right \rangle
N_{3},  \label{eqGAW4}
\end{equation}

$v)$ of generalized $AW(5)$-type if they satify%
\begin{equation}
N_{4}=a_{1}N_{1}+b_{1}N_{2},  \label{eqGAW5}
\end{equation}

$vi)$ of generalized $AW(6)$-type if they satify%
\begin{equation}
N_{4}=a_{2}N_{1}+b_{2}N_{3},  \label{eqGAW6}
\end{equation}

$vii)$ of generalized $AW(7)$-type if they satisfy%
\begin{equation}
N_{4}=a_{3}N_{2}+b_{3}N_{3},  \label{eqGAW7}
\end{equation}%
where $a_{i},b_{i}$ $(1\leq i\leq 3)$ are non-zero real valued
differentiable functions.
\end{definition}

\begin{remark}
\bigskip We use notation $GAW(k)$-type for curves of generalized $AW(k)$%
-type.
\end{remark}

Geometrically, a curve of $GAW(k)$-type is a curve whose fifth derivative's
normal part is either zero or linearly dependent with one or two of its
previous derivatives' normal parts.

Firstly, we give the following proposition:

\begin{proposition}
The osculating order of a Frenet curve of any $GAW(k)$-type can not be
bigger than or equal to $5$.
\end{proposition}

\begin{proof}
Let $\gamma :I\subseteq 
\mathbb{R}
\rightarrow \mathbb{E}^{n}$ be a Frenet curve of osculating order $d$. If $%
\gamma $ is of any $GAW(k)$-type, since none of $N_{i}$ $(1\leq i\leq 3)$
contains a component in the direction of $v_{5}$, we find $\mu _{5}=\kappa
_{1}\kappa _{2}\kappa _{3}\kappa _{4}=0$. This concludes $d\leq 4$, which
completes the proof.
\end{proof}

Using equations \ref{Normalparts} and Definition \ref{definitionAW(k)}, we
obtain the following main theorem:

\begin{theorem}
\label{maintheorem}Let $\gamma $ be a unit speed Frenet curve of osculating
order $d\leq 4$ in $n$-dimensional Euclidean space $\mathbb{E}^{n}$. Then $%
\gamma $ is

$i)$ \ of \ $GAW(1)$-type if and only if%
\begin{equation*}
\mu _{2}=\mu _{3}=\mu _{4}=0,
\end{equation*}

$ii)$ \ of \ $GAW(2)$-type if and only if%
\begin{equation*}
\mu _{4}=0,
\end{equation*}%
\begin{equation*}
\kappa _{1}\kappa _{2}\mu _{2}-\kappa _{1}^{\prime }\mu _{3}=0,
\end{equation*}

$iii)$ \ of $GAW(3)$-type if and only if%
\begin{equation*}
\mu _{3}=\mu _{4}=0,
\end{equation*}

$iv)$ of $GAW(4)$-type if and only if%
\begin{equation*}
\lambda _{2}\mu _{3}-\lambda _{3}\mu _{2}=0,
\end{equation*}%
\begin{equation*}
\lambda _{2}\mu _{4}-\lambda _{4}\mu _{2}=0,
\end{equation*}

$v)$ \ of $GAW(5)$-type if and only if%
\begin{equation*}
\mu _{2}=a_{1}\kappa _{1}+b_{1}\kappa _{1}^{\prime },
\end{equation*}%
\begin{equation*}
\mu _{3}=b_{1}\kappa _{1}\kappa _{2},
\end{equation*}%
\begin{equation*}
\mu _{4}=0,
\end{equation*}

$vi)$ of $GAW(6)$-type if and only if%
\begin{equation*}
\mu _{2}=a_{2}\kappa _{1}+b_{2}\lambda _{2},
\end{equation*}%
\begin{equation*}
\mu _{3}=b_{2}\lambda _{3},
\end{equation*}%
\begin{equation*}
\mu _{4}=b_{2}\lambda _{4},
\end{equation*}

$vii)$ of $GAW(7)$-type if and only if%
\begin{equation*}
\mu _{2}=a_{3}\kappa _{1}^{\prime }+b_{3}\lambda _{2},
\end{equation*}%
\begin{equation*}
\mu _{3}=a_{3}\kappa _{1}\kappa _{2}+b_{3}\lambda _{3},
\end{equation*}%
\begin{equation*}
\mu _{4}=b_{3}\lambda _{4}.
\end{equation*}
\end{theorem}

\begin{proof}
i) Let $\gamma $ be of $GAW(1)$-type. Then, from equations (\ref{Normalparts}%
) and Definition \ref{definitionAW(k)} , we have $N_{4}=\mu _{2}v_{2}+\mu
_{3}v_{3}+\mu _{4}v_{4}=0$. Since $v_{2},$ $v_{3}$ and $v_{4}$ are linearly
independent, we get $\mu _{2}=\mu _{3}=\mu _{4}=0$. The sufficiency is
trivial.

ii) Let $\gamma $ be of $GAW(2)$-type. If we calculate $\left \Vert
N_{2}\right \Vert ^{2}$ and $\left \langle N_{2},N_{4}\right \rangle $, by the
use of equations (\ref{Normalparts}) and (\ref{eqGAW(2)}), we obtain%
\begin{equation*}
\lbrack (\kappa _{1}^{\prime })^{2}+\kappa _{1}^{2}\kappa _{2}^{2}](\mu
_{2}v_{2}+\mu _{3}v_{3}+\mu _{4}v_{4})=(\kappa _{1}^{\prime }\mu _{2}+\kappa
_{1}\kappa _{2}\mu _{3})(\kappa _{1}^{\prime }v_{2}+\kappa _{1}\kappa
_{2}v_{3}).
\end{equation*}%
Since $v_{2},$ $v_{3}$ and $v_{4}$ are linearly independent, we find $\mu
_{4}=0$ and $\kappa _{1}\kappa _{2}\mu _{2}-\kappa _{1}^{\prime }\mu _{3}=0$%
. Conversely, if $\mu _{4}=0$ and $\kappa _{1}\kappa _{2}\mu _{2}-\kappa
_{1}^{\prime }\mu _{3}=0$, one can easily show that equation (\ref{eqGAW(2)}%
) is satisfied.

iii) Let $\gamma $ be of $GAW(3)$-type. We get $\left \Vert N_{1}\right \Vert
^{2}=\kappa _{1}^{2}$ and $\left \langle N_{1},N_{4}\right \rangle =\kappa
_{1}\mu _{2}$. So, if we write these equations in (\ref{eqGAW3}), we have%
\begin{equation*}
\kappa _{1}^{2}(\mu _{2}v_{2}+\mu _{3}v_{3}+\mu _{4}v_{4})=\kappa _{1}\mu
_{2}(\kappa _{1}v_{2}).
\end{equation*}

Thus, $\mu _{3}=\mu _{4}=0$. Converse theorem is clear.

iv) Let $\gamma $ be of $GAW(4)$-type. We can easily calculate $\left \Vert
N_{3}\right \Vert ^{2}=\lambda _{2}^{2}+\lambda _{3}^{2}+\lambda _{4}^{2}$
and $\left \langle N_{3},N_{4}\right \rangle =\lambda _{2}\mu _{2}+\lambda
_{3}\mu _{3}+\lambda _{4}\mu _{4}$. So equation (\ref{eqGAW4}) gives us%
\begin{equation*}
(\lambda _{2}^{2}+\lambda _{3}^{2}+\lambda _{4}^{2})(\mu _{2}v_{2}+\mu
_{3}v_{3}+\mu _{4}v_{4})=(\lambda _{2}\mu _{2}+\lambda _{3}\mu _{3}+\lambda
_{4}\mu _{4})(\lambda _{2}v_{2}+\lambda _{3}v_{3}+\lambda _{4}v_{4}).
\end{equation*}%
Hence, we can write%
\begin{equation}
(\lambda _{2}^{2}+\lambda _{3}^{2}+\lambda _{4}^{2})\mu _{2}=(\lambda
_{2}\mu _{2}+\lambda _{3}\mu _{3}+\lambda _{4}\mu _{4})\lambda _{2},
\label{eq1}
\end{equation}%
\begin{equation}
(\lambda _{2}^{2}+\lambda _{3}^{2}+\lambda _{4}^{2})\mu _{3}=(\lambda
_{2}\mu _{2}+\lambda _{3}\mu _{3}+\lambda _{4}\mu _{4})\lambda _{3},
\label{eq2}
\end{equation}%
\begin{equation}
(\lambda _{2}^{2}+\lambda _{3}^{2}+\lambda _{4}^{2})\mu _{4}=(\lambda
_{2}\mu _{2}+\lambda _{3}\mu _{3}+\lambda _{4}\mu _{4})\lambda _{4}.
\label{eq3}
\end{equation}

If we multiply (\ref{eq1}) with $\lambda _{3}$ and use equation (\ref{eq2}),
we find $\lambda _{2}\mu _{3}-\lambda _{3}\mu _{2}=0$. Multiplying (\ref{eq1}%
) with $\lambda _{4}$ and using equation (\ref{eq3}), we have $\lambda
_{2}\mu _{4}-\lambda _{4}\mu _{2}=0$. Conversely, it is easy to show that
equation (\ref{eqGAW4}) is satisfied if $\lambda _{2}\mu _{3}-\lambda
_{3}\mu _{2}=0$ and $\lambda _{2}\mu _{4}-\lambda _{4}\mu _{2}=0$.

v) Let $\gamma $ be of $GAW(5)$-type. Then, in view of equations (\ref%
{eqGAW5}) and (\ref{Normalparts}), we can write%
\begin{equation*}
\mu _{2}v_{2}+\mu _{3}v_{3}+\mu _{4}v_{4}=a_{1}(\kappa
_{1}v_{2})+b_{1}(\kappa _{1}^{\prime }v_{2}+\kappa _{1}\kappa _{2}v_{3})%
\text{,}
\end{equation*}%
which gives us $\mu _{2}=a_{1}\kappa _{1}+b_{1}\kappa _{1}^{\prime }$, $\mu
_{3}=b_{1}\kappa _{1}\kappa _{2}$ and $\mu _{4}=0$. Conversely, if these
last three equations are satisfied, one can show that $%
N_{4}=a_{1}N_{1}+b_{1}N_{2}$.

vi) Let $\gamma $ be of $GAW(6)$-type. By definition, we have $%
N_{4}=a_{2}N_{1}+b_{2}N_{3}$, that is,%
\begin{equation*}
\mu _{2}v_{2}+\mu _{3}v_{3}+\mu _{4}v_{4}=a_{2}(\kappa
_{1}v_{2})+b_{2}(\lambda _{2}v_{2}+\lambda _{3}v_{3}+\lambda _{4}v_{4}).
\end{equation*}%
Since $v_{2},$ $v_{3}$ and $v_{4}$ are linearly independent, we can write%
\begin{equation*}
\mu _{2}=a_{2}\kappa _{1}+b_{2}\lambda _{2},
\end{equation*}%
\begin{equation*}
\mu _{3}=b_{2}\lambda _{3},
\end{equation*}%
\begin{equation*}
\mu _{4}=b_{2}\lambda _{4}.
\end{equation*}%
Conversely, if these last equations are satisfied, then we easily show that $%
N_{4}=a_{2}N_{1}+b_{2}N_{3}$.

vii) Let $\gamma $ be of $GAW(7)$-type. Then using equations (\ref%
{Normalparts}) and (\ref{eqGAW7}), we obtain%
\begin{equation*}
\mu _{2}v_{2}+\mu _{3}v_{3}+\mu _{4}v_{4}=a_{3}(\kappa _{1}^{\prime
}v_{2}+\kappa _{1}\kappa _{2}v_{3})+b_{3}(\lambda _{2}v_{2}+\lambda
_{3}v_{3}+\lambda _{4}v_{4})\text{.}
\end{equation*}%
Thus%
\begin{equation*}
\mu _{2}=a_{3}\kappa _{1}^{\prime }+b_{3}\lambda _{2},
\end{equation*}%
\begin{equation*}
\mu _{3}=a_{3}\kappa _{1}\kappa _{2}+b_{3}\lambda _{3},
\end{equation*}%
\begin{equation*}
\mu _{4}=b_{3}\lambda _{4}.
\end{equation*}%
Conversely, let $\gamma $ be a curve satisfying the last three equations. It
is easily found that $N_{4}=a_{3}N_{2}+b_{3}N_{3}.$
\end{proof}

From now on, we consider Frenet curves whose first curvature $\kappa _{1}$
is a constant. We give curvature conditions of such a curve to be of $GAW(k)$%
-type. We can state following propositions:

\begin{proposition}
Let $\gamma :I\subseteq 
\mathbb{R}
\rightarrow \mathbb{E}^{n}$ be a unit speed Frenet curve of osculating order 
$d\leq 4$ with $\kappa _{1}=$constant. Then $\gamma $ is of $GAW(1)$-type if
and only if it is a straight line or a circle.
\end{proposition}

\begin{proof}
Let $\gamma $ be of $GAW(1)$-type. Since $\kappa _{1}=$constant, using (\ref%
{nus}) and Theorem \ref{maintheorem}, we find%
\begin{equation}
\mu _{2}=-3\kappa _{1}\kappa _{2}\kappa _{2}^{\prime }=0,  \label{equ1}
\end{equation}%
\begin{equation}
\mu _{3}=-\kappa _{1}^{3}\kappa _{2}-\kappa _{1}\kappa _{2}^{3}+\kappa
_{1}\kappa _{2}^{\prime \prime }-\kappa _{1}\kappa _{2}\kappa _{3}^{2}=0,
\label{equ2}
\end{equation}%
\begin{equation}
\mu _{4}=2\kappa _{1}\kappa _{2}^{\prime }\kappa _{3}+\kappa _{1}\kappa
_{2}\kappa _{3}^{\prime }=0.  \label{equ3}
\end{equation}

If $\kappa _{1}=0$, then $\gamma $ is a straight line and above three
equations are satisfied. Let $\kappa _{1}$ be a non-zero constant. If $%
\kappa _{2}=0$, then $\gamma $ is a circle and equations (\ref{equ1}), (\ref%
{equ2}) and (\ref{equ3}) are satisfied again. Assume that $\kappa _{2}\neq 0$%
. Then (\ref{equ1}) gives us $\kappa _{2}^{\prime }=0$, that is, $\kappa
_{2} $ is a constant. In this case, from equation (\ref{equ2}), we get $%
(\kappa _{1}^{2}+\kappa _{2}^{2}+\kappa _{3}^{2})=0$, which means $\kappa
_{1}=\kappa _{2}=\kappa _{3}=0$. This is a contradiction. So $\kappa _{2}=0$.

Conversely, let $\gamma $ be a straight line or a circle. Thus $\kappa
_{1}=0 $; or $\kappa _{1}=$constant and $\kappa _{2}=0$. So $\mu _{2}=\mu
_{3}=\mu _{4}=0$, which completes the proof.
\end{proof}

\begin{proposition}
Let $\gamma :I\subseteq 
\mathbb{R}
\rightarrow \mathbb{E}^{n}$ be a unit speed Frenet curve of osculating order 
$d\leq 4$ with $\kappa _{1}=$constant. Then $\gamma $ is of $GAW(2)$-type if
and only if

i) it is a straight line; or

ii) it is a circle; or

iii) it is a helix of order $3$ or $4$.
\end{proposition}

\begin{proof}
Let $\gamma $ be of $GAW(2)$-type. Since $\kappa _{1}=$constant, using (\ref%
{nus}) and Theorem \ref{maintheorem}, we obtain%
\begin{equation}
\mu _{4}=2\kappa _{1}\kappa _{2}^{\prime }\kappa _{3}+\kappa _{1}\kappa
_{2}\kappa _{3}^{\prime }=0,  \label{equa1}
\end{equation}%
\begin{equation}
\kappa _{1}\kappa _{2}\left( -3\kappa _{1}\kappa _{2}\kappa _{2}^{\prime
}\right) =0.  \label{equa2}
\end{equation}%
One can easily see that $\kappa _{2}$ and $\kappa _{3}$ must be constants.
Thus, $\gamma $ can be a straight line, a circle or a helix of order $3$ or $%
4$. Conversely, if $\gamma $ is one of these curves, the proof is clear
using Theorem \ref{maintheorem}.
\end{proof}

\begin{proposition}
Let $\gamma :I\subseteq 
\mathbb{R}
\rightarrow \mathbb{E}^{n}$ be a unit speed Frenet curve of osculating order 
$d\leq 4$ with $\kappa _{1}=$constant. Then $\gamma $ is of $GAW(3)$-type if
and only if

i) it is a straight line; or

ii) it is a circle; or

iii) it is a Frenet curve of osculating order $3$ satisfying the second
order non-linear ODE 
\begin{equation*}
\kappa _{2}^{\prime \prime }=\kappa _{2}(\kappa _{1}^{2}+\kappa _{2}^{2})%
\text{; or}
\end{equation*}

iv) it is a Frenet curve of osculating order $4$ with%
\begin{equation*}
\kappa _{2}=\frac{c}{\sqrt{\kappa _{3}}}
\end{equation*}%
and its third curvature satisfies the second order non-linear ODE%
\begin{equation}
\kappa _{3}^{\prime \prime }-\frac{3(\kappa _{3}^{\prime })^{2}}{2\kappa _{3}%
}+2\kappa _{3}(\kappa _{1}^{2}+\kappa _{3}^{2})+2c^{2}=0,  \label{ODE}
\end{equation}%
and where $c>0$ is an arbitrary constant.
\end{proposition}

\begin{proof}
Let $\gamma $ be of $GAW(3)$-type. Since $\kappa _{1}=$constant, using (\ref%
{nus}) and Theorem \ref{maintheorem}, we have%
\begin{equation}
\mu _{3}=-\kappa _{1}^{3}\kappa _{2}-\kappa _{1}\kappa _{2}^{3}+\kappa
_{1}\kappa _{2}^{\prime \prime }-\kappa _{1}\kappa _{2}\kappa _{3}^{2}=0,
\label{equ4}
\end{equation}%
\begin{equation}
\mu _{4}=2\kappa _{1}\kappa _{2}^{\prime }\kappa _{3}+\kappa _{1}\kappa
_{2}\kappa _{3}^{\prime }=0.  \label{equ5}
\end{equation}%
If $d=1$ or $d=2$, we obtain line and circle cases, both of which do not
conradict above two equations. Let $d=3.$ Then $\kappa _{1}=$constant$>0$, $%
\kappa _{2}>0$ and $\kappa _{3}=0$. (\ref{equ5}) is satisfied directly and (%
\ref{equ4}) gives us%
\begin{equation*}
\kappa _{2}^{\prime \prime }=\kappa _{2}(\kappa _{1}^{2}+\kappa _{2}^{2}),
\end{equation*}%
which is a second order non-linear ODE. Now, let $d=4$. Thus, $\kappa _{1}=$%
constant$>0$, $\kappa _{2}>0$ and $\kappa _{3}>0$. If we solve (\ref{equ5}),
we find 
\begin{equation}
\kappa _{2}=\frac{c}{\sqrt{\kappa _{3}}},  \label{equ6}
\end{equation}%
where $c>0$ is an arbitrary constant. Then 
\begin{equation*}
\kappa _{2}^{\prime }=\frac{-c\kappa _{3}^{\prime }}{2\kappa _{3}^{3/2}},
\end{equation*}%
\begin{equation}
\kappa _{2}^{\prime \prime }=c.\left[ \frac{3(\kappa _{3}^{\prime })^{2}}{%
4\kappa _{3}^{5/2}}-\frac{\kappa _{3}^{\prime \prime }}{2\kappa _{3}^{3/2}}%
\right] .  \label{equ7}
\end{equation}%
If we multiply equation (\ref{equ4}) with $\frac{\kappa _{2}}{\kappa _{1}}$,
using (\ref{equ6}) and (\ref{equ7}), we obtain the second order non-linear
ODE (\ref{ODE}). Conversely, if $\gamma $ is one of these curves, one can
show that $\mu _{3}=\mu _{4}=0.$
\end{proof}

\begin{proposition}
Let $\gamma :I\subseteq 
\mathbb{R}
\rightarrow \mathbb{E}^{n}$ be a unit speed Frenet curve of osculating order 
$d\leq 4$ with $\kappa _{1}=$constant. Then $\gamma $ is of $GAW(4)$-type if
and only if

i) it is a straight line; or

ii) it is a circle; or

iii) it is a Frenet curve of osculating order $3$ satisfying the second
order non-linear ODE 
\begin{equation}
3\kappa _{2}(\kappa _{2}^{\prime })^{2}=(\kappa _{1}^{2}+\kappa _{2}^{2}) 
\left[ \kappa _{2}^{\prime \prime }-\kappa _{2}(\kappa _{1}^{2}+\kappa
_{2}^{2})\right] \text{; or}  \label{eq10}
\end{equation}

iv) it is a Frenet curve of osculating order $4$ with%
\begin{equation}
\kappa _{2}^{2}\kappa _{3}=c.\left( \kappa _{1}^{2}+\kappa _{2}^{2}\right)
^{3/2}  \label{e11}
\end{equation}%
and its curvatures satify%
\begin{equation*}
3\kappa _{2}(\kappa _{2}^{\prime })^{2}=(\kappa _{1}^{2}+\kappa _{2}^{2}) 
\left[ \kappa _{2}^{\prime \prime }-\kappa _{2}(\kappa _{1}^{2}+\kappa
_{2}^{2}+\kappa _{3}^{2})\right] .
\end{equation*}%
Here, $c>0$ is an arbitrary constant.
\end{proposition}

\begin{proof}
Let $\gamma $ be of $GAW(4)$-type. Since $\kappa _{1}=$constant, using (\ref%
{lamdas}), (\ref{nus}) and Theorem \ref{maintheorem}, we find%
\begin{equation}
(-\kappa _{1}^{3}-\kappa _{1}\kappa _{2}^{2})(-\kappa _{1}^{3}\kappa
_{2}-\kappa _{1}\kappa _{2}^{3}+\kappa _{1}\kappa _{2}^{\prime \prime
}-\kappa _{1}\kappa _{2}\kappa _{3}^{2})-(\kappa _{1}\kappa _{2}^{\prime
})(-3\kappa _{1}\kappa _{2}\kappa _{2}^{\prime })=0,  \label{equ8}
\end{equation}%
\begin{equation}
(-\kappa _{1}^{3}-\kappa _{1}\kappa _{2}^{2})(2\kappa _{1}\kappa
_{2}^{\prime }\kappa _{3}+\kappa _{1}\kappa _{2}\kappa _{3}^{\prime
})-(\kappa _{1}\kappa _{2}\kappa _{3})(-3\kappa _{1}\kappa _{2}\kappa
_{2}^{\prime })=0.  \label{equ9}
\end{equation}%
(\ref{equ8}) and (\ref{equ9}) give us%
\begin{equation}
3\kappa _{2}(\kappa _{2}^{\prime })^{2}=(\kappa _{1}^{2}+\kappa _{2}^{2}) 
\left[ \kappa _{2}^{\prime \prime }-\kappa _{2}(\kappa _{1}^{2}+\kappa
_{2}^{2}+\kappa _{3}^{2})\right] \text{,}  \label{equ10}
\end{equation}%
\begin{equation}
(2\kappa _{1}^{2}-\kappa _{2}^{2})\kappa _{3}\kappa _{2}^{\prime }+\kappa
_{2}(\kappa _{1}^{2}+\kappa _{2}^{2})\kappa _{3}^{\prime }=0\text{.}
\label{equ11}
\end{equation}%
Now, if $\kappa _{1}=0$, then $\gamma $ is a straight line and equations (%
\ref{equ8}) and (\ref{equ9}) are satisfied. Let $\kappa _{1}$ be a non-zero
constant. If $\kappa _{2}=0$, then $\gamma $ is a circle. Let $\kappa _{2}>0$
and $\kappa _{3}=0$. Then, from equation (\ref{equ10}), we obtain (\ref{eq10}%
). Now, let $d=4$. Then, using equation (\ref{equ11}), we can write%
\begin{equation*}
\int \frac{(2\kappa _{1}^{2}-\kappa _{2}^{2})}{\kappa _{2}(\kappa
_{1}^{2}+\kappa _{2}^{2})}d\kappa _{2}+\int \frac{1}{\kappa _{3}}d\kappa
_{3}=\ln c\text{,}
\end{equation*}%
where $c>0$ is an arbitrary constant. Remember that $\kappa _{1}>0$ is a
constant. So we find%
\begin{equation*}
2\ln (\kappa _{2})-\frac{3}{2}\ln (\kappa _{1}^{2}+\kappa _{2}^{2})+\ln
(\kappa _{3})=\ln c\text{,}
\end{equation*}%
which gives us (\ref{e11}). Furthermore, $\gamma $ must also satisfy (\ref%
{equ10}). Conversely, if $\gamma $ is one of the curves above, we can show
that (\ref{equ8}) and (\ref{equ9}) are satisfied.
\end{proof}

\begin{proposition}
\label{propGAW5}Let $\gamma :I\subseteq 
\mathbb{R}
\rightarrow \mathbb{E}^{n}$ be a unit speed Frenet curve of osculating order 
$d\leq 4$ with $\kappa _{1}=$constant. Then $\gamma $ is of $GAW(5)$-type if
and only if

i) it is a straight line; or

ii) it is a Frenet curve of osculating order $3$ with 
\begin{equation*}
\kappa _{2}\neq \text{constant}
\end{equation*}
and 
\begin{equation*}
\kappa _{2}^{\prime \prime }\neq \kappa _{2}\left( \kappa _{1}^{2}+\kappa
_{2}^{2}\right) \text{; or}
\end{equation*}

iii) it is a Frenet curve of osculating order $4$ with%
\begin{equation*}
\kappa _{2}\neq \text{constant, }\kappa _{3}\neq \text{constant,}
\end{equation*}%
\begin{equation*}
\kappa _{2}^{\prime \prime }\neq \kappa _{2}\left( \kappa _{1}^{2}+\kappa
_{2}^{2}+\kappa _{3}^{2}\right)
\end{equation*}%
and%
\begin{equation*}
\kappa _{2}=\frac{c}{\sqrt{\kappa _{3}}},
\end{equation*}%
where $c>0$ is an arbitrary constant.
\end{proposition}

\begin{proof}
Let $\gamma $ be of $GAW(5)$-type. Since $\kappa _{1}=$constant, by the use
of Theorem \ref{maintheorem} and equations (\ref{nus}), we have%
\begin{equation}
-3\kappa _{1}\kappa _{2}\kappa _{2}^{\prime }=a_{1}\kappa _{1},  \label{e1}
\end{equation}%
\begin{equation}
-\kappa _{1}^{3}\kappa _{2}-\kappa _{1}\kappa _{2}^{3}+\kappa _{1}\kappa
_{2}^{\prime \prime }-\kappa _{1}\kappa _{2}\kappa _{3}^{2}=b_{1}\kappa
_{1}\kappa _{2},  \label{e2}
\end{equation}%
\begin{equation}
2\kappa _{1}\kappa _{2}^{\prime }\kappa _{3}+\kappa _{1}\kappa _{2}\kappa
_{3}^{\prime }=0.  \label{e3}
\end{equation}%
If $d=1$, then $\gamma $ is a straight line and above equations are
satisfied. If $d=2$, then $\gamma $ is a circle. From (\ref{e1}), we find $%
a_{1}=0$, which contradicts the definition. Now, let $d=3$. Then, using (\ref%
{e1}) and (\ref{e2}), we find%
\begin{equation*}
a_{1}=-3\kappa _{2}\kappa _{2}^{\prime }\text{,}
\end{equation*}%
\begin{equation*}
b_{1}=\frac{\kappa _{2}^{\prime \prime }}{\kappa _{2}}-\kappa
_{1}^{2}-\kappa _{2}^{2}\text{.}
\end{equation*}%
Since $a_{1}$ and $b_{1}$ are non-zero functions, then $\kappa _{2}\neq $%
constant and $\kappa _{2}^{\prime \prime }\neq \kappa _{2}\left( \kappa
_{1}^{2}+\kappa _{2}^{2}\right) $. Finally, let $d=4$. Then, equation (\ref%
{e3}) gives us%
\begin{equation}
\kappa _{2}=\frac{c}{\sqrt{\kappa _{3}}},  \label{e8}
\end{equation}%
where $c>0$ is an arbitrary constant. In this case, from (\ref{e1}) and (\ref%
{e2}), we find%
\begin{equation}
a_{1}=-3\kappa _{2}\kappa _{2}^{\prime }\text{,}  \label{e9}
\end{equation}%
\begin{equation}
b_{1}=\frac{\kappa _{2}^{\prime \prime }}{\kappa _{2}}-\kappa
_{1}^{2}-\kappa _{2}^{2}-\kappa _{3}^{2}\text{.}  \label{e6}
\end{equation}%
Thus, (\ref{e8}) and (\ref{e9}) give us 
\begin{equation}
\kappa _{2}\neq \text{constant, }\kappa _{3}\neq \text{constant.}
\label{e10}
\end{equation}%
Also, from (\ref{e6}), we can write%
\begin{equation}
\kappa _{2}^{\prime \prime }\neq \kappa _{2}\left( \kappa _{1}^{2}+\kappa
_{2}^{2}+\kappa _{3}^{2}\right) .  \label{e7}
\end{equation}%
Converse proposition is trivial.
\end{proof}

\begin{proposition}
Let $\gamma :I\subseteq 
\mathbb{R}
\rightarrow \mathbb{E}^{n}$ be a unit speed Frenet curve of osculating order 
$d\leq 4$ with $\kappa _{1}=$constant. Then $\gamma $ is of $GAW(6)$-type if
and only if

i) it is a straight line; or

ii) it is a circle; or

iii) it is a Frenet curve of osculating order $3$ with 
\begin{equation*}
\kappa _{2}\neq \text{constant,}
\end{equation*}%
\begin{equation*}
\kappa _{2}^{\prime \prime }\neq \kappa _{2}\left( \kappa _{1}^{2}+\kappa
_{2}^{2}\right)
\end{equation*}%
and 
\begin{equation*}
\kappa _{2}^{\prime \prime }\neq \kappa _{2}\left( \kappa _{1}^{2}+\kappa
_{2}^{2}\right) +\frac{3\kappa _{2}\left( \kappa _{2}^{\prime }\right) ^{2}}{%
\kappa _{1}^{2}+\kappa _{2}^{2}}\text{; or}
\end{equation*}

iv) it is a Frenet curve of osculating order $4$ with%
\begin{equation*}
\kappa _{2}\neq \text{constant,}
\end{equation*}%
\begin{equation*}
\kappa _{2}\neq \frac{c}{\sqrt{\kappa _{3}}},
\end{equation*}%
\begin{equation*}
\frac{2\kappa _{2}^{\prime }}{\kappa _{2}}+\frac{\kappa _{3}^{\prime }}{%
\kappa _{3}}=\frac{\kappa _{2}^{\prime \prime }}{\kappa _{2}^{\prime }}-%
\frac{\kappa _{2}}{\kappa _{2}^{\prime }}(\kappa _{1}^{2}+\kappa
_{2}^{2}+\kappa _{3}^{2})
\end{equation*}%
and 
\begin{equation}
\kappa _{2}^{\prime \prime }\neq \kappa _{2}\left( \kappa _{1}^{2}+\kappa
_{2}^{2}+\kappa _{3}^{2}\right) +\frac{3\kappa _{2}\left( \kappa
_{2}^{\prime }\right) ^{2}}{\kappa _{1}^{2}+\kappa _{2}^{2}}\text{.}
\label{A1}
\end{equation}%
Here, $c>0$ is an arbitrary constant.
\end{proposition}

\begin{proof}
Let $\gamma $ be of $GAW(6)$-type. Since $\kappa _{1}=$constant, by the use
of equations (\ref{lamdas}), (\ref{nus}) and Theorem \ref{maintheorem}, we
have%
\begin{equation}
-3\kappa _{1}\kappa _{2}\kappa _{2}^{\prime }=a_{2}\kappa _{1}+b_{2}\left(
-\kappa _{1}^{3}-\kappa _{1}\kappa _{2}^{2}\right) ,  \label{EQ1}
\end{equation}%
\begin{equation}
-\kappa _{1}^{3}\kappa _{2}-\kappa _{1}\kappa _{2}^{3}+\kappa _{1}\kappa
_{2}^{\prime \prime }-\kappa _{1}\kappa _{2}\kappa _{3}^{2}=b_{2}\kappa
_{1}\kappa _{2}^{\prime },  \label{EQ2}
\end{equation}%
\begin{equation}
2\kappa _{1}\kappa _{2}^{\prime }\kappa _{3}+\kappa _{1}\kappa _{2}\kappa
_{3}^{\prime }=b_{2}\kappa _{1}\kappa _{2}\kappa _{3}.  \label{EQ3}
\end{equation}%
If $\kappa _{1}=0$, then $\gamma $ is a straight line. Let $d=2$. Then $%
\gamma $ is a circle and from (\ref{EQ1}), we obtain%
\begin{equation*}
a_{2}-b_{2}\kappa _{1}^{2}=0,
\end{equation*}%
which is satisfied for some $a_{2}$, $b_{2}$ non-zero differentiable
functions. (\ref{EQ2}) and (\ref{EQ3}) are also satisfied. Now, let $d=3$.
Then we have%
\begin{equation}
-3\kappa _{2}\kappa _{2}^{\prime }=a_{2}-b_{2}\left( \kappa _{1}^{2}+\kappa
_{2}^{2}\right) ,  \label{EQ4}
\end{equation}%
\begin{equation}
\kappa _{2}^{\prime \prime }-\kappa _{2}(\kappa _{1}^{2}+\kappa
_{2}^{2})=b_{2}\kappa _{2}^{\prime }.  \label{EQ5}
\end{equation}%
Thus $\kappa _{2}$ can not be constant. So (\ref{EQ4}) and (\ref{EQ5}) give
us 
\begin{equation*}
b_{2}=\frac{\kappa _{2}^{\prime \prime }}{\kappa _{2}^{\prime }}-\frac{%
\kappa _{2}}{\kappa _{2}^{\prime }}(\kappa _{1}^{2}+\kappa _{2}^{2}),
\end{equation*}%
\begin{equation*}
a_{2}=-3\kappa _{2}\kappa _{2}^{\prime }+\frac{\kappa _{2}^{\prime \prime }}{%
\kappa _{2}^{\prime }}(\kappa _{1}^{2}+\kappa _{2}^{2})-\frac{\kappa _{2}}{%
\kappa _{2}^{\prime }}(\kappa _{1}^{2}+\kappa _{2}^{2})^{2},
\end{equation*}%
both of which must be non-zero. Finally, let $d=4$. From (\ref{EQ2}), $%
\kappa _{2}\neq $constant. In this case, by the use of (\ref{EQ1}), (\ref%
{EQ2}) and (\ref{EQ3}), we obtain%
\begin{equation}
b_{2}=\frac{2\kappa _{2}^{\prime }}{\kappa _{2}}+\frac{\kappa _{3}^{\prime }%
}{\kappa _{3}}=\frac{\kappa _{2}^{\prime \prime }}{\kappa _{2}^{\prime }}-%
\frac{\kappa _{2}}{\kappa _{2}^{\prime }}(\kappa _{1}^{2}+\kappa
_{2}^{2}+\kappa _{3}^{2}),  \label{EQ6}
\end{equation}%
\begin{equation}
a_{2}=-3\kappa _{2}\kappa _{2}^{\prime }+\frac{\kappa _{2}^{\prime \prime }}{%
\kappa _{2}^{\prime }}(\kappa _{1}^{2}+\kappa _{2}^{2})-\frac{\kappa _{2}}{%
\kappa _{2}^{\prime }}(\kappa _{1}^{2}+\kappa _{2}^{2})(\kappa
_{1}^{2}+\kappa _{2}^{2}+\kappa _{3}^{2}).  \label{A2}
\end{equation}%
Thus, from equation (\ref{EQ6}), we have%
\begin{equation*}
\kappa _{2}\neq \frac{c}{\sqrt{\kappa _{3}}},
\end{equation*}%
where $c>0$ is an arbitrary constant. We also have (\ref{A1}) from (\ref{A2}%
).

Converse proposition is done easily.
\end{proof}

\begin{proposition}
Let $\gamma :I\subseteq 
\mathbb{R}
\rightarrow \mathbb{E}^{n}$ be a unit speed Frenet curve of osculating order 
$d\leq 4$ with $\kappa _{1}=$constant. Then $\gamma $ is of $GAW(7)$-type if
and only if

i) it is a straight line; or

ii) it is a Frenet curve of osculating order $3$ satisying 
\begin{equation*}
\kappa _{2}\neq \text{constant,}
\end{equation*}%
\begin{equation*}
\kappa _{2}^{\prime \prime }\neq \frac{3\kappa _{2}\left( \kappa
_{2}^{\prime }\right) ^{2}}{\kappa _{1}^{2}+\kappa _{2}^{2}}+\kappa
_{2}(\kappa _{1}^{2}+\kappa _{2}^{2})\text{; or}
\end{equation*}

iv) it is a Frenet curve of osculating order $4$ satisfying%
\begin{equation*}
\kappa _{2}\neq \text{constant,}
\end{equation*}%
\begin{equation*}
\kappa _{2}\neq \frac{c}{\sqrt{\kappa _{3}}},
\end{equation*}%
\begin{equation*}
\frac{3\kappa _{2}\kappa _{2}^{\prime }}{\kappa _{1}^{2}+\kappa _{2}^{2}}=%
\frac{2\kappa _{2}^{\prime }}{\kappa _{2}}+\frac{\kappa _{3}^{\prime }}{%
\kappa _{3}},
\end{equation*}%
\begin{equation*}
\kappa _{2}^{\prime \prime }\neq \frac{3\kappa _{2}\left( \kappa
_{2}^{\prime }\right) ^{2}}{\kappa _{1}^{2}+\kappa _{2}^{2}}+\kappa
_{2}(\kappa _{1}^{2}+\kappa _{2}^{2}+\kappa _{3}^{2}),
\end{equation*}%
where $c>0$ is an arbitrary constant.
\end{proposition}

\begin{proof}
Let $\gamma $ be of $GAW(7)$-type. If we use equations (\ref{lamdas}), (\ref%
{nus}) and Theorem \ref{maintheorem}, we obtain%
\begin{equation}
-3\kappa _{1}\kappa _{2}\kappa _{2}^{\prime }=b_{3}(-\kappa _{1}^{3}-\kappa
_{1}\kappa _{2}^{2}),  \label{b1}
\end{equation}%
\begin{equation}
-\kappa _{1}^{3}\kappa _{2}-\kappa _{1}\kappa _{2}^{3}+\kappa _{1}\kappa
_{2}^{\prime \prime }-\kappa _{1}\kappa _{2}\kappa _{3}^{2}=a_{3}\kappa
_{1}\kappa _{2}+b_{3}\kappa _{1}\kappa _{2}^{\prime },  \label{b2}
\end{equation}%
\begin{equation}
2\kappa _{1}\kappa _{2}^{\prime }\kappa _{3}+\kappa _{1}\kappa _{2}\kappa
_{3}^{\prime }=b_{3}\kappa _{1}\kappa _{2}\kappa _{3}.  \label{b3}
\end{equation}%
If $d=1$, $\gamma $ is a straight line. Let $d=2$. Then, from (\ref{b1}), we
find $\kappa _{1}=0$. This is a contradiction. Let $d=3$. Then, using (\ref%
{b1}), $\kappa _{2}$ can not be contant. By the use of (\ref{b1}) and (\ref%
{b2}), we get%
\begin{equation}
b_{3}=\frac{3\kappa _{2}\kappa _{2}^{\prime }}{\kappa _{1}^{2}+\kappa
_{2}^{2}},  \label{b4}
\end{equation}%
\begin{equation*}
a_{3}=\frac{\kappa _{2}^{\prime \prime }}{\kappa _{2}}-\frac{3\kappa
_{2}\left( \kappa _{2}^{\prime }\right) ^{2}}{\kappa _{2}\left( \kappa
_{1}^{2}+\kappa _{2}^{2}\right) }-(\kappa _{1}^{2}+\kappa _{2}^{2}),
\end{equation*}%
both of which are non-zero differentiable functions. Again, equation (\ref%
{b4}) requires $\kappa _{2}$ is not a contant. We also have 
\begin{equation*}
\kappa _{2}^{\prime \prime }\neq \frac{3\kappa _{2}\left( \kappa
_{2}^{\prime }\right) ^{2}}{\kappa _{1}^{2}+\kappa _{2}^{2}}+\kappa
_{2}(\kappa _{1}^{2}+\kappa _{2}^{2}).
\end{equation*}%
Now, let $d=4$. Then, using equations (\ref{b1}), (\ref{b2}) and (\ref{b3}),
we obtain%
\begin{equation}
b_{3}=\frac{3\kappa _{2}\kappa _{2}^{\prime }}{\kappa _{1}^{2}+\kappa
_{2}^{2}}=\frac{2\kappa _{2}^{\prime }}{\kappa _{2}}+\frac{\kappa
_{3}^{\prime }}{\kappa _{3}},  \label{b5}
\end{equation}%
\begin{equation*}
a_{3}=\frac{\kappa _{2}^{\prime \prime }}{\kappa _{2}}-\frac{3\kappa
_{2}\left( \kappa _{2}^{\prime }\right) ^{2}}{\kappa _{2}\left( \kappa
_{1}^{2}+\kappa _{2}^{2}\right) }-(\kappa _{1}^{2}+\kappa _{2}^{2}+\kappa
_{3}^{2}),
\end{equation*}%
which give us%
\begin{equation*}
\kappa _{2}\neq \text{constant,}
\end{equation*}%
\begin{equation*}
\kappa _{2}\neq \frac{c}{\sqrt{\kappa _{3}}},
\end{equation*}%
\begin{equation*}
\kappa _{2}^{\prime \prime }\neq \frac{3\kappa _{2}\left( \kappa
_{2}^{\prime }\right) ^{2}}{\kappa _{1}^{2}+\kappa _{2}^{2}}+\kappa
_{2}(\kappa _{1}^{2}+\kappa _{2}^{2}+\kappa _{3}^{2}).
\end{equation*}%
Here, $c>0$ is an arbitrary constant.

Converse proposition is trivial.
\end{proof}

\bigskip

\noindent Kadri ARSLAN

\noindent Department of Mathematics,

\noindent Uludag University,

\noindent G\"{o}r\"{u}kle Campus,16059 Bursa, TURKEY

\noindent Email: arslan@uludag.edu.tr

\  \  \  \  \  \  \  \  \bigskip

\noindent \c{S}aban G\"{U}VEN\c{C}

\noindent Department of Mathematics,

\noindent Balikesir University,

\noindent \c{C}a\u{g}\i \c{s}, 10145 Balikesir, TURKEY

\noindent Email: sguvenc@balikesir.edu.tr

\end{document}